\documentclass[a4paper]{llncs}
\usepackage[utf8]{inputenc}

\usepackage{amsmath,amssymb}
\usepackage{enumerate,enumitem}
\usepackage{graphicx,subfigure}
\usepackage{textcomp,gensymb}
\usepackage{color,url}
\usepackage{cite}

 \let\doendproof\endproof
\renewcommand\endproof{~\hfill\qed\doendproof}

\title{Drawing graphs with vertices and edges in convex position}
\author{Ignacio Garc\'ia-Marco \and Kolja Knauer}

\author{Ignacio Garc\'ia-Marco\inst{1}\fnmsep\thanks{supported by ANR project CompA (project number: ANR-13-BS02-0001-01)} \and Kolja Knauer\inst{2}\fnmsep\thanks{supported by ANR EGOS grant ANR-12-JS02-002-01 and PEPS grant EROS}}
\institute{LIP, ENS Lyon - CNRS - UCBL - INRIA, Universit\'e de Lyon
UMR 5668, Lyon, France
\\\email{ignacio.garcia-marco@ens-lyon.fr} \and Aix-Marseille
Universit\'e, CNRS, LIF UMR 7279, Marseille,
France\\\email{kolja.knauer@lif.univ-mrs.fr}}

\begin{document}

\maketitle

\begin{abstract}
A graph has strong convex dimension $2$ if it admits a
straight-line drawing in the plane such that its vertices form a convex set 
and the midpoints of its edges also constitute a convex
set. Halman, Onn, and Rothblum conjectured that graphs of
strong convex dimension $2$ are planar and therefore have at most
$3n-6$ edges. We prove that all such graphs have indeed at most $2n-3$
edges, while on the other hand we present an infinite family of non-planar
graphs of strong convex dimension $2$. We give lower bounds on
the maximum number of edges a graph of strong convex dimension $2$
can have and discuss several natural variants of this graph class. Furthermore, we apply our
methods to obtain new results about large convex sets in
Minkowski sums of planar point sets -- a topic that has been of interest in
recent years.
\end{abstract}

\section{Introduction}

A point set $X\subseteq\mathbb{R}^2$ is \emph{(strictly) convex}
if every point in $X$ is a vertex of the convex hull of $X$. A
point set $X$ is said to be \emph{weakly convex} if $X$
lies on the boundary of its convex hull. A \emph{drawing} of a graph
$G$ is a mapping~$f:V(G)\to\mathbb{R}^2$ such that edges
are straight line segments connecting vertices and neither midpoints
of edges, nor vertices, nor midpoints and vertices coincide. Through
most of the paper we will not distinguish between (the elements of)
a graph and their drawings.

For $i,j\in\{s,w,a\}$ we define $\mathcal{G}_i^j$ as the class of
graphs
admitting a drawing such that the set of vertices is $\begin{cases} \mbox{strictly convex} &\mbox{if } i=s \\
\mbox{weakly convex} & \mbox{if } i=w \\
\mbox{arbitrary} & \mbox{if } i=a \end{cases}$ \ and the midpoints of edges constitute a $\begin{cases} \mbox{strictly convex} &\mbox{if } j=s \\
\mbox{weakly convex} & \mbox{if } j=w \\
\mbox{arbitrary} & \mbox{if } j=a \end{cases}$\  set. Further,
we define~$g_i^j(n)$ to be the maximum number of edges an $n$-vertex
graph in $\mathcal{G}_i^j$ can~have.

Clearly, all $\mathcal{G}_i^j$ are closed under taking subgraphs and $\mathcal{G}_s^a=\mathcal{G}_w^a=\mathcal{G}_a^a$ is the class of all graphs.

\subsubsection*{Previous results and related problems:}
Motivated by a special class of convex optimization
problems~\cite{Onn-04}, Halman, Onn, and Rothblum~\cite{Hal-07}
studied drawings of graphs in $\mathbb{R}^d$ with similar
constraints as described above. In particular, in their language a
graph has convex dimension $2$ if and only if it is in
$\mathcal{G}_a^s$ and strong convex dimension $2$ if and only if it
is in $\mathcal{G}_s^s$. They show that all trees and cycles are in
$\mathcal{G}_s^s$, while
$K_4\in\mathcal{G}_a^s\setminus\mathcal{G}_s^s$ and
$K_{2,3}\notin\mathcal{G}_a^s$. Moreover, they show that~$n\leq
g_s^s(n)\leq 5n-8$. Finally, they conjecture that all graphs in
$\mathcal{G}_s^s$ are planar and thus $g_s^s(n)\leq 3n-6$.

The problem of computing or bounding $g_a^s(n)$ and $g_s^s(n)$ was rephrased and
generalized in the setting of convex subsets of
Minkowski sums of planar point sets by Eisenbrand et
al.~\cite{UpBound2008} and then regarded as a problem of
computational geometry in its own right. We introduce this
setting and give an overview of known results before explaining
its relation to the original graph drawing problem.

Given two point sets $A, B\subseteq \mathbb{R}^d$ their
\emph{Minkowski sum} $A+B$ is defined as~$\{a+b\mid a\in A, b\in
B\} \subseteq \mathbb{R}^d$. We define $M(m,n)$ as the largest cardinality of a convex set $X \subseteq A + B$, for $A$ and $B$ planar point
sets with $|A|=m$ and~$|B|=n$. In~\cite{UpBound2008} it was shown
that $M(m,n)\in O(m^{2/3}n^{2/3}+m+n)$. This upper bound was complemented by B\'{\i}lka et
al.~\cite{LowBound2010} with
an asymptotically matching lower bound, even under the assumption that $A$ itself is
convex, i.e., $M(m,n)\in \Theta(m^{2/3}n^{2/3}+m+n)$.
Notably, the lower bound works also for the case $A=B$ non-convex, as shown by
Swanepoel and Valtr~\cite[Proposition 4]{SV2010}. In~\cite{Tiw-14} Tiwary gives an
upper bound of~$O((m+n)\log(m+n))$ for the largest cardinality of a convex
set~$X \subseteq A + B$, for $A$ and $B$ planar
convex point sets with~$|A|=m$ and~$|B|=n$. Determining the
asymptotics in this case remains an open question.

As first observed in~\cite{UpBound2008}, the graph drawing problem of Halman et al. is related to the largest
cardinality of a convex set $X \subset A + A$, for $A$
some planar point set. In fact, from $X$ and $A$ one can deduce a
graph $G\in\mathcal{G}_a^s$ on vertex set $A$, with an edge $aa'$
for all $a\neq a'$ with $a+a'\in X$. The midpoint of the edge $aa'$
then just is $\frac{1}{2}(a+a')\in \frac{1}{2}X\subset \frac{1}{2}A + \frac{1}{2}A$. Conversely, from any
$G\in\mathcal{G}_a^s$ one can construct $X$ and $A$ as desired. The
only trade-off in this translation are the pairs of the form $aa$,
which are not taken into account by the graph-model, because they
correspond to vertices. Hence, they do not play a role from the
purely asymptotic point of view. Thus, the results
of~\cite{UpBound2008,LowBound2010,SV2010} yield $g_a^s(n) =
\Theta(n^{4/3})$. Conversely, the bounds for $g_s^s(n)$ obtained
in~\cite{Hal-07} give that the largest cardinality of a convex set~$X \subseteq A + A$, for $A$ a planar convex point
set with $|A|=n$ is in $\Theta(n)$.

% Results from this area
% yield the asymptotic behaviour of $g_a^s(n)$:
% In~\cite{SV2010}, Swanepoel and Valtr studied the value of $E_d(n)$,
% the maximum number of pairs that can be selected from a set of $n$
% points in $\mathbb{R}^d$ such that the midpoints of these pairs are
% in convex position. When $n = 2$, applying some results of
% Eisenbrand et al.~\cite{UpBound2008} and B\'{\i}lka et
% al.~\cite{LowBound2010}, they derived that $E_2(n) =
% \Theta(n^{4/3})$. It suffices to observe that whenever we have a
% pair of distinct elements, its midpoint is the midpoint of the edge
% connecting them and, hence, $E_2(n) - n \leq g_a^s(n) \leq E_2(n)$,
% which implies that $g_a^s(n) = \Theta(n^{4/3})$.
% \comment{no habia tambien un paper de Hans Raj Tiwary?}
%

\subsubsection*{Our results:}
In this paper we study the set of graph classes defined in the
introduction. We extend the list of properties of point sets
considered in earlier works with \emph{weak} convexity. We
completely determine the inclusion relations on the resulting
classes. We prove that $\mathcal{G}_s^s$ contains non-planar graphs,
which disproves a conjecture of Halman et al.~\cite{Hal-07}, and
that $\mathcal{G}_s^w$ contains cubic graphs, while we believe 
is false for $\mathcal{G}_s^s$. We give new bounds for the
parameters~$g_i^j(n)$: we show that $g_s^w(n)=2n-3$, which is an
upper bound for $g_s^s(n)$ and therefore improves the upper bound of
$3n-6$ conjectured by Halman et al.~\cite{Hal-07}. Furthermore we
show that $\lfloor\frac{3}{2}(n-1)\rfloor\leq g_s^s(n)$.

For the relation with Minkowski sums we show that the largest
cardinality of a weakly convex set $X \subseteq A +
A$, for $A$ some convex planar point set of $|A|=n$, is $2n$ and of a
strictly convex set is between $\frac{3}{2}n$ and $2n-2$. 

The results for weak convexity are the first non-trivial precise formulas in this area.

\bigskip

A preliminary version of this paper has been published in conference proceedings\cite{Gar-15}.

\section{Graph drawings}

Given a graph $G$ drawn in the plane with straight line segments as
edges, we denote by $P_V$ the convex hull of its set of vertices and by
$P_E$ the convex hull of the set of midpoints of its edges. Clearly, unless $V=\emptyset$, $P_E$
is strictly contained in $P_V$.

\subsection{Inclusions of classes}
We show that most of the classes defined in the introduction coincide and determine the exact set of inclusions among the remaining classes.

\begin{theorem}\label{thm:clases}
We have
$\mathcal{G}_s^s=\mathcal{G}_w^s\subsetneq\mathcal{G}_s^w\subsetneq\mathcal{G}_w^w=\mathcal{G}_a^w=\mathcal{G}_s^a=\mathcal{G}_w^a=\mathcal{G}_a^a$
and $\mathcal{G}_s^s\subsetneq\mathcal{G}_a^s \subsetneq
\mathcal{G}_w^w$. Moreover, there is no inclusion relationship
between $\mathcal{G}_a^s$ and $\mathcal{G}_s^w$. See
Figure~\ref{fig:clases} for an illustration.\end{theorem}

 \begin{figure}[htb]
  \centering
  \includegraphics{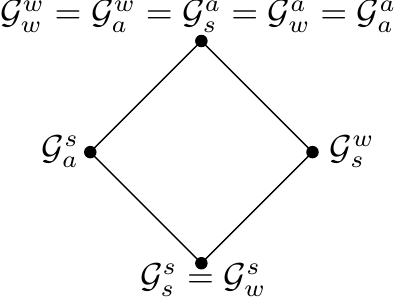}
  \caption{Inclusions and identities among the classes $\mathcal{G}_i^j$.}
  \label{fig:clases}
 \end{figure}

\begin{proof}
Let us begin by proving that $\mathcal{G}_s^s=\mathcal{G}_w^s$, the
inclusion $\mathcal{G}_s^s \subseteq \mathcal{G}_w^s$ is obvious. Take
$G \in \mathcal G_w^s$ drawn in the required way. Since the midpoints of the edges
form a convex set, there exists $\delta > 0$ such that moving every vertex by at
most $< \delta$ in any direction, the set of midpoints of the edges remains 
strictly convex. More precisely, whenever there are vertices $z_1,\ldots,z_k$
in the interior of the segment connecting two vertices $x,y$, we perform the following steps, see Figure~\ref{fig:localtrans}:

 \begin{figure}[htb]
  \centering
  \includegraphics{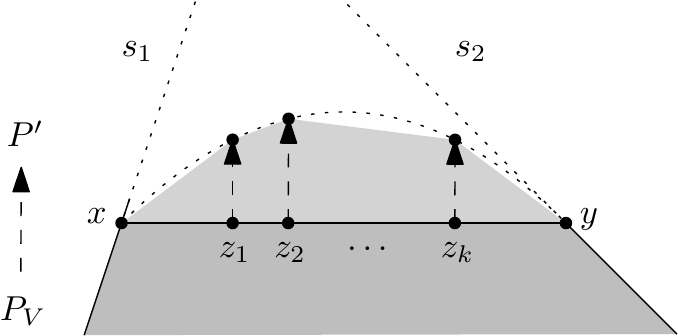}
  \caption{The local modifications to prove $\mathcal{G}_s^s \supseteq \mathcal{G}_w^s$.}
  \label{fig:localtrans}
 \end{figure}

We assume without loss of generality
that $x$ is drawn at $(0,0)$, $y$ is drawn at $(1,0)$ and that $P_V$
is entirely contained in the closed halfplane $\{(a,b)\, \vert \, b
\leq 0\}$. We consider the two adjacent edges to $xy$ in the boundary of $P_V$ and denote by 
$s_1, s_2 \in \mathbb{R} \cup \{\pm \infty\}$ their
slopes.
Now we take $\epsilon:\, 0 < \epsilon < {\rm min}\{\delta, \vert
s_1 \vert, \vert s_2 \vert\}$, we observe that $P' := P_V
\cup \{(a,b) \, \vert \, 0 \leq a \leq 1$ and $0 \leq b \leq
\epsilon a (1-a)\}$ is a convex set. Then, for all $i \in \{1,\ldots,k\}$,
if $z_i$ is drawn at $(\lambda_i,0)$ with $0 < \lambda_i < 1$, we
translate $z_i$ to the point $(\lambda_i, \epsilon \lambda_i (1 -
\lambda_i))$. We observe that the point $z_i$ has been moved a
distance~$< \epsilon/4 < \delta$ and, then, the set of midpoints of
edges is still convex. Moreover, now
$z_1,\ldots,z_k$ are vertices of $P'$. Repeating this
argument when necessary we get that $G \in \mathcal G_s^{s}$.

To prove the strict inclusion
$\mathcal{G}_s^s\subsetneq\mathcal{G}_s^w$ we show that the graph
$K_4 - e$, i.e., the graph obtained from removing an edge $e$ from
the complete graph $K_4$ belongs to $\mathcal{G}_s^w$ but not to
$\mathcal{G}_s^s$. Indeed, if we take $x_0,x_1,x_2,x_3$ the $4$
vertices of $K_4 - e$ and assume that $e = x_2x_3$, it suffices to
draw $x_0 = (1,0)$, $x_1 = (0,0)$, $x_2 = (0,1)$ and $x_3 = (2,1)$
to get that $K_4 - e \in \mathcal G_s^w$. See Figure~\ref{fig:K4-e} for an illustration.

\begin{figure}[htb]
  \centering
  \includegraphics{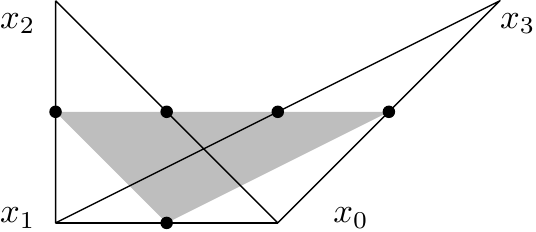}
  \caption{A drawing proving $K_4 - e \in \mathcal G_s^w$.}
  \label{fig:K4-e}
 \end{figure}

Let us now prove that $K_4 - e
\notin \mathcal G_s^s$. To that end, we assume that the set of vertices $\{x_0,x_1,x_2,x_3\}$ is in convex position. By means of an affine transformation we may assume that $x_0,x_1,x_2,x_3$ are drawn at the points
$(1,0), (0,0), (0,1)$ and $(a,b),$ with~$a,b >
0$ respectively. The fact that $\{x_0,x_1,x_2,x_3\}$ is in convex position implies that $a + b
> 1$. If $x_ix_{i+1\ {\rm mod}\ 4}$ is an edge for all $i \in
\{0,1,2,3\}$, then clearly the set of midpoints is not convex because the midpoints
of $x_0x_2$ and $x_1x_3$ are in the convex hull of the midpoints of the other $4$
edges. So, assume that $x_2 x_3$ is not an edge, , i.e., the drawing is like in Figure~\ref{fig:K4-e}. So the midpoints of
the edges are in positions $m_{01} = (0,1/2), m_{12} = (1/2, 0)$, $m_{02} =
(1/2, 1/2)$, $m_{13} = (a/2,b/2)$, $m_{03} = (a/2, (b+1)/2)$. (We will generally denote midpoints in this fashion.) If $m_{01},
m_{12}, m_{02}, m_{13}$ are in convex position, then we deduce that $a < 1$
or $b < 1$ but not both, since otherwise $x_3$ would be in the convex hull of $x_0,x_1,x_2$. 
However, if $a < 1$, then $m_{03}$ belongs to
the convex hull of $\{m_{01},m_{12},m_{02},m_{13}\}$, and if $b < 1$, then $m_{13}$
belongs to the convex hull of $\{m_{01},m_{12},m_{02},m_{03}\}$. Hence, we again
have that the set of midpoints is not convex and we conclude that $K_4 - e \notin
\mathcal{G}_s^s$.

 The
strict inclusion $\mathcal{G}_s^w\subsetneq\mathcal{G}_a^a$ comes as
a direct consequence of Theorem~\ref{2n-2}.

Let us see that every graph belongs to $\mathcal{G}_w^w$, for this
purpose it suffices to show that $K_n \in \mathcal{G}_w^w$. Drawing the vertices in the points 
with coordinates $(0,0)$, and $(1,2^i)$ for~$i \in
\{1,\ldots,n-1\}$ gives the result. Indeed, the midpoints of all the edges lie either in the vertical line 
$x = 1/2$ or in $x = 1$. The choice of $y$-coordinates ensures that in the line $x=1$, no midpoints coincide with other midpoints nor vertices. Hence, we clearly have that
$\mathcal{G}_w^w=\mathcal{G}_a^w=\mathcal{G}_s^a=\mathcal{G}_w^a=\mathcal{G}_a^a$.

The strictness in the inclusions $\mathcal{G}_s^s\subsetneq\mathcal{G}_a^s
\subsetneq \mathcal{G}_w^w$ comes from the fact that $g_a^s =
\Theta(n^{4/3})$~\cite{UpBound2008,LowBound2010,SV2010} and that,
 $g_s^s(n) \leq g_s^w(n) \leq 2n-3$ by Theorem~\ref{2n-2}. This
 also proves that $\mathcal{G}_a^s \not\subset \mathcal{G}_{s}^w$.

To prove that $\mathcal{G}_s^w \not\subset \mathcal{G}_{a}^s$ it
suffices to consider the complete bipartite graph $K_{2,3}$. Indeed,
if $\{x_1,x_2,x_3\}$, $\{y_1,y_2\}$ is the vertex partition, it
suffices to draw $x_1,x_2,x_3$ in $(0,0)$, $(4,0)$, $(3,2)$,
respectively, and $y_1, y_2$ in $(1,1)$, $(4,1)$, respectively, to
get that $K_{2,3} \in \mathcal{G}_s^w$. See Figure~\ref{fig:K23} for an illustration.

\begin{figure}[htb]
  \centering
  \includegraphics[width=.8\textwidth]{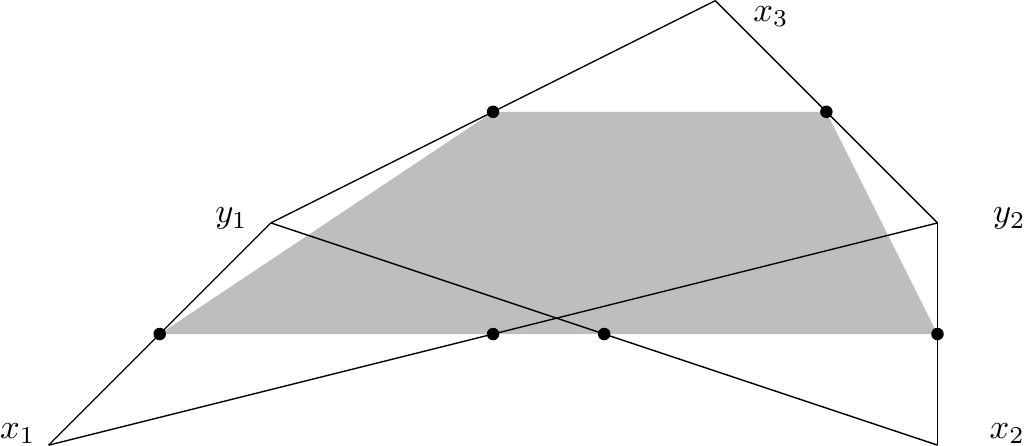}
  \caption{A drawing proving $K_{2,3} \in \mathcal G_s^w$. }
  \label{fig:K23}
 \end{figure}

Finally, $K_{2,3} \notin
\mathcal{G}_a^s$ was already shown in~\cite{Hal-07}.
%It follows
%from the easy to prove fact that if $p_1,p_2,q_1,q_2,q_3 \in
%\mathbb{R}^2$, then the $6$ points in $\mathbb{R}^2$ of the form
%$p_i + q_j$ are never in convex position

%
% \comment{ No se si de todas maneras vale la pena demostrar que la
% familia de Theorem~\ref{thm:3over2} tampoco esta en
% $\mathcal{G}_s^s$... }
\end{proof}

\subsection{Bounds on numbers of edges}

In this section, we show that $\lfloor\frac{3}{2}(n-1)\rfloor\leq g_s^s(n)\leq g_s^w(n)=2n-3$.

Whenever $V$ is weakly convex, for every vertex $x$, one can order
the neighbors of $x$ according to their clockwise appearance around
the border of $P_V$ starting at $x$. If in this order the neighbors
of $x$ are $y_1,\ldots,y_k$, then we say that~$xy_2,\ldots,xy_{k-1}$
are the {\it interior edges of $x$}. Non-interior edges of $x$ are
called {\it exterior edges of $x$}. Clearly, any vertex has at most
two exterior edges. A \emph{vertex~$v$ sees an edge $e$} if the
straight-line segment connecting $v$ and the midpoint $m_e$ of~$e$
does not intersect the interior of~$P_E$, recall that $P_E$ is the convex hull of the midpoints.

\begin{lemma}\label{sees2}
 If $G\in\mathcal{G}_s^w$, then no vertex sees its interior edges. In particular, any
vertex sees at most $2$ incident edges.
\end{lemma}
\begin{proof}
Assume that there exists a vertex $x$ seeing an interior edge
$xu_i$. Take $u_1, u_k$ such that $xu_1, xu_k$ are the exterior
edges of $x$. We consider the induced graph $G'$ with vertex set $V'
= \{v,u_1,u_i,u_k\}$ and denote by $E'$ its corresponding edge set.
Clearly $P_{V'} \subset P_V$ and $P_{E'} \subset P_E$, so $x$ sees
$xu_i$ in $P_{E'}$. Moreover, $xu_i$ is still an interior edge of
$x$ in $G'$. Denote by $m_j$ the midpoint of the edge $vu_j$,
for~$j\in\{1,i,k\}$. Since $x$ sees $xu_i$, the closed halfplane
supported by the line passing through $m_1,m_k$ containing $x$ also
contains $m_i$.

However, since $P_{V'}$ is strictly convex $u_i$ and $x$ are separated by the line passing through $u_1,u_k$. This is a
contradiction because $m_j = (u_j + x)/2$. See Figure~\ref{fig:lemma1}.

 \begin{figure}[htb]
  \centering
  \includegraphics{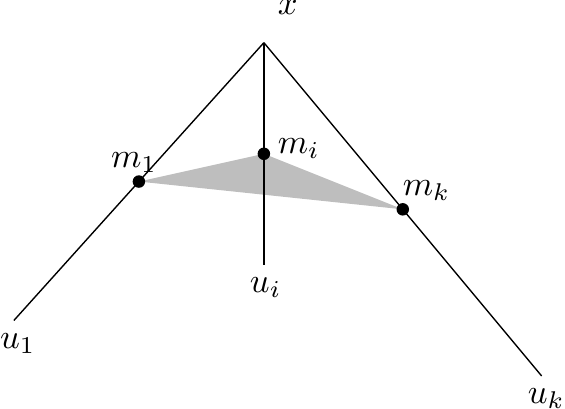}
  \caption{The construction in Lemma~\ref{sees2}}
  \label{fig:lemma1}
 \end{figure}

\end{proof}

\begin{theorem}\label{2n-2}
 If a graph $G\in\mathcal{G}_s^w$ has $n$ vertices, then it has at most $2n-3$ edges, i.e., $g_s^w(n)\leq 2n-3$.
\end{theorem}
\begin{proof}

Take $G \in \mathcal{G}_s^w$. Since the midpoints of the edges form a weakly convex set,
 every edge has to be seen by at least one of its
vertices. Lemma~\ref{sees2} guarantees that interior edges cannot be
seen. Hence, no edge can be interior to both endpoints. This proves
that $G$ has at most $2n$ edges.

We improve this bound by showing that at least three edges are
exterior to both of their endpoints, i.e., are counted twice in the
above estimate. During the proof let us call such edges \emph{doubly
exterior}.

Since deleting leafs only decreases the ratio of vertices and
edges, we can assume that $G$ has no leafs. Since $2>2n-3$ implies $n\leq 1$ and in this case our statement is clearly true, 
we can also assume that $G$ has at least three edges. For an edge $e$, we denote
by $H_e^+$ and $H_e^-$ the open halfplanes supported by the line
containing $e$. We claim that whenever an edge $e=xy$ is an interior
edge of $x$, then $H_e^+\cup\{x\}$ contains a
doubly exterior edge. This follows by induction on the number of
vertices in $H_e^+\cap P_V$.  If there is a single vertex $z \in H_e^+ \cap P_V$, then 
$xz$ is an exterior edge of $x$ because $xy$ is interior to $x$. Moreover, by convexity of $V$ and since $z$ is the only vertex in $H_e^+$ the edge $xz$ is also exterior to $z$, so it is doubly exterior. We assume now that there is more
than one vertex in $H_e^+ \cap P_V$.
Since $e$ is interior to $x$, there is
an edge $f=xz$ contained in $H_e^+\cup\{x\}$ and exterior of $x$. If
$f$ is doubly exterior we are done. Otherwise, we set $H_f^+$ the
halfplane supported by the line containing~$f$ and not containing
$y$. We claim that $(H_f^+\cup\{z\}) \cap V \subset
(H_e^+\cup\{x\}) \cap V$. Indeed, if there is a point $v \in
(H_f^+\cup\{z\}) \cap V$ but not in $H_e^+\cup\{x\}$, then $x$ is
in the interior of the triangle with vertices $v, y, z \in V$, a
contradiction. Thus, $(H_f^+\cup\{z\}) \cap V$ is contained in
$(H_e^+\cup\{x\}) \cap V$, in
particular, since $(H_f^+\cup\{z\}) \cap V$ does not contain $x$ the inclusion is strict. By induction, we can guarantee
that $(H_e^+\cup\{x\}) \cap P_V$ contains a doubly exterior edge.

Note that an analogous argument yields that $H_e^-\cup\{x\}$, contains a doubly exterior edge if $e$ is an interior edge of $x$.

Applying this argument to any edge $e$ which is not doubly exterior
gives already two doubly exterior edges $f,g$ contained in
$H_e^+\cup\{x\}$ and $H_e^-\cup\{x\}$, respectively. Choose an
endpoint $z$ of $f$, which is not an endpoint of $g$, 
which is possible since we have minimum degree at least two. Let~$h=zw$ be
the other exterior edge of $z$. If $h$ is doubly exterior we are
done. Otherwise, none of $H_h^+\cup\{w\}$ and $H_h^-\cup\{w\}$
contains $f$ because $z \notin H_h^+$ and $z \notin H_h^-$; moreover
one of $H_h^+\cup\{w\}$ and $H_h^-\cup\{w\}$ does not contain $g$.
Thus, there must be a third doubly exterior edge.
\end{proof}

\begin{definition}\label{def:Ln}
For every $n \geq 2$, we denote by $L_n$ the graph consisting of two
paths $P=(u_1, \ldots,u_{\lfloor\frac{n}{2}\rfloor})$ and $Q=(v_1,
\ldots, v_{\lceil\frac{n}{2}\rceil})$ and the edges $u_1v_1$ and
$u_iv_{i-1}$ and $u_{j-1}v_j$ for $1<i\leq\lfloor\frac{n}{2}\rfloor$
and $1<j\leq\lceil\frac{n}{2}\rceil$. We observe that $L_n$ has
$2n-3$ edges.
\end{definition}

\begin{theorem}\label{thm:2n-3lb}
 For all $n\geq 2$ we have $L_n\in\mathcal{G}_s^w$, i.e., $g_s^w(n)\geq 2n-3$.
\end{theorem}
\begin{proof}

For every $k \geq 1$ we construct a drawing showing 
$L_{4k+2}\in\mathcal{G}_s^w$ (the result for other values of $n$
follows by suppressing degree $2$ vertices). We take $0 < \epsilon_0
< \epsilon_1 < \cdots < \epsilon_{2k}$ and set $\delta_j := \sum_{i
= j}^{2k} \epsilon_i$ for all $j \in \{0,\ldots,2k\}$. We consider
the graph $G$ with vertices $r_i = (i, \delta_{2i}), r_i' = (i,
-\delta_{2i})$ for $i \in \{0,\ldots,k\}$ and $\ell_i = (-i,
\delta_{2i-1}), \ell_i' = (-i, -\delta_{2i-1})$ for $i \in
\{1,\ldots,k\}$; and edge set
$$\{r_0 r_0'\} \cup \{r_i \ell_i, r_i \ell_i', r_i' \ell_i,
r_i'\ell_i' \, \vert \, 1 \leq i \leq k\} \cup \{r_{i-1} \ell_i,
r_{i-1} \ell_i', r_{i-1}' \ell_i, r_{i-1}'\ell_i' \, \vert \, 1 \leq
i \leq k\}.$$
\begin{figure}[htb]
  \centering
  \includegraphics{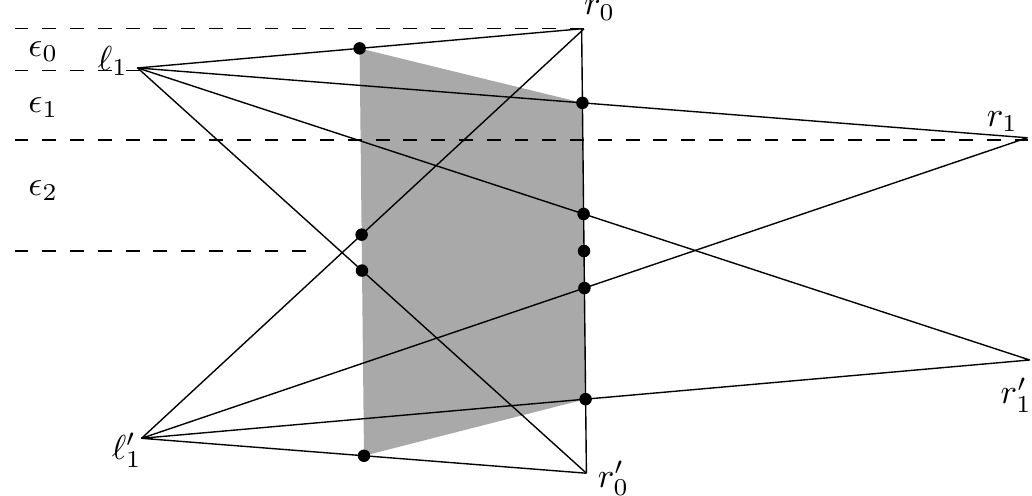}
  \caption{The graph $L_6$ is in $\mathcal{G}_s^w$.}
  \label{fig:2n-3base}
 \end{figure}
See Figure~\ref{fig:2n-3base} for an illustration of the final
drawing. The choice of $\epsilon_i$ forming an increasing sequence yields directly that the set of vertices is strictly convex.
Moreover, the midpoints of the edges all lie on the vertical lines $x = 0$ and~$x = -1/2$; thus they
form a weakly convex set. It is straightforward to verify that
the constructed graph is $L_{4k+2}$.

\end{proof}

We observe that $L_4 = K_4 - e$ and that $L_4$ is a subgraph of $L_n$ for all $n \geq 4$. As we proved in Theorem~\ref{thm:clases},
$K_4 - e$ does not belong to $\mathcal G_{s}^s$. Hence, for all $n \geq 4$ we have that $L_n \notin \mathcal G_s^s$.

\medskip

Theorem~\ref{2n-2} together with Theorem~\ref{thm:2n-3lb} yield the exact value of $g_s^w(n) = 2n-3$. Moreover, since $\mathcal G_s^s \subset \mathcal G_s^w$, from
Theorem~\ref{2n-2} we also deduce the upper bound $g_s^s(n) \leq 2n-3$. The rest of this section is devoted to provide a lower bound for $g_s^s(n)$.

\begin{definition}
For every odd $n \geq 3$, we denote by $B_n$ the graph obtained from
identifying a $C_3$ and $\frac{n-3}{2}$ copies of $C_4$ altogether
identified along a single edge $uv$. We observe that $B_n$ has
$\frac{3}{2}(n-1)$ edges and deleting a degree $2$ vertex from
$B_{n}$ one obtains an $(n-1)$-vertex graph with
$\frac{3}{2}(n-2)-\frac{1}{2}$ edges.
\end{definition}

\begin{theorem}\label{thm:3/2n-2lb}
For all odd $n\geq 3$ we have $B_n\in\mathcal{G}_s^s$, i.e.,
$g_s^s(n)\geq \lfloor\frac{3}{2}(n-1)\rfloor$.
\end{theorem}
\begin{proof}
Let $n\geq 3$ be such that $n-3$ is divisible by $4$ (if $n-3$ is
not divisible by $4$, then $B_n$ is an induced subgraph of
$B_{n+1}$). We will first draw $B_n$ in an unfeasible way and then
transform it into another one proving $B_n\in\mathcal{G}_s^s$.

 See Figure~\ref{fig:2nover3} for an illustration of the final drawing.
\begin{figure}[htb]
  \centering
  \includegraphics[width=\textwidth]{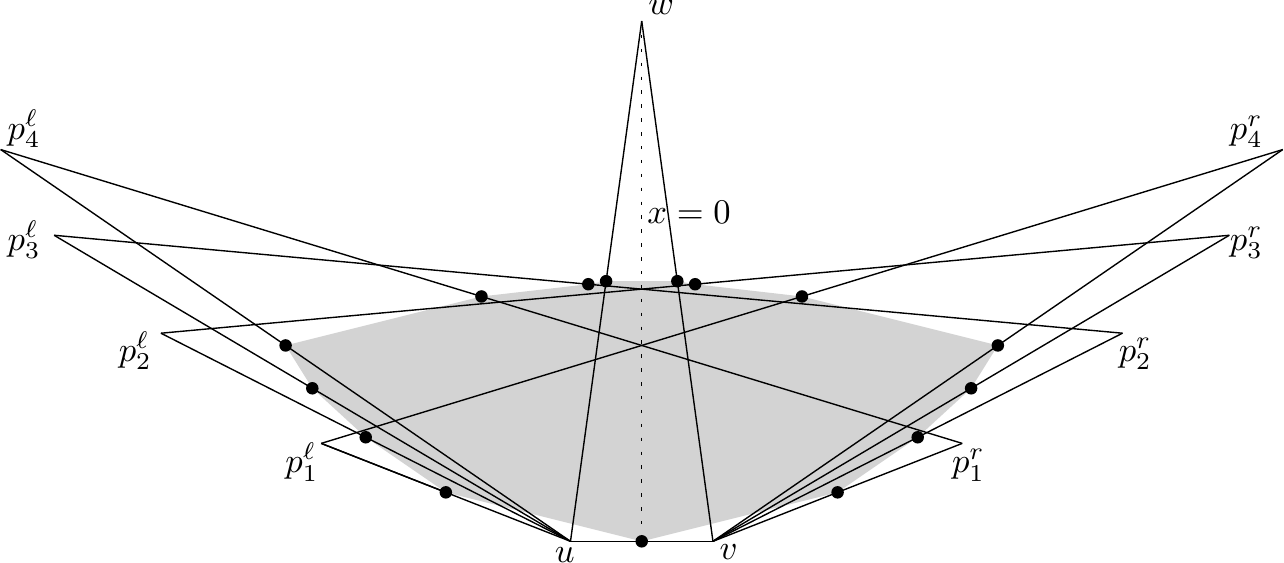}
  \caption{The graph $B_{11}$ is in $\mathcal{G}_s^s$.}
  \label{fig:2nover3}
 \end{figure}

We draw the $C_3=(uvw)$ as an isosceles triangle with horizontal
base $uv$. Let $u=(-1,0)$, $v=(1,0)$, and $w=(0,\frac{n-1}{2})$.
There are $n-3$ remaining points. Draw one half of them on
coordinates $p^{\ell}_i=(-1-i,i)$ for $1\leq i\leq\frac{n-3}{2}$ and
the other half mirrored along the $y$-axis, i.e., $p^r_i=(1+i,i)$
for $1\leq i\leq\frac{n-3}{2}$.

Now we add all edges  $p^{\ell}_iu$ (left edges), $p^r_iv$ (right
edges), for $1\leq i\leq\frac{n-3}{2}$ and edges of the form
$p^{\ell}_ip^r_{\frac{n-3}{2}+1-i}$ (diagonal edges) for all $1\leq
i\leq\frac{n-3}{2}$.

We observe that the points $p^{\ell}_i$ and $u$ lie on the line $x +
y = - 1$, the points~$p^{r}_i$ and $v$ lie on the line $x - y = 1$
and all midpoints of diagonal edges have $y$-coordinate
$\frac{n-1}{4}$. In order to bring the set of vertices and the set of midpoints of edges into 
convex sets, we simultaneously decrease the $y$-coordinates of
points $p^{\ell}_{\frac{n-3}{2} + 1 - i}, p^r_{\frac{n-3}{2} + 1 -
i}$ by $2^i \epsilon$ for $i \in \{1,\ldots,\frac{n-3}{2}\}$ for a
sufficiently small value $\epsilon
> 0$. Finally, we conveniently decrease the
$y$-coordinate of $w$ to get a drawing witnessing that
$B_n\in\mathcal{G}_s^s$.

\end{proof}

\subsection{Further members of $\mathcal{G}_s^s$ and $\mathcal{G}_s^w$}

We show that there are non-planar graphs in $\mathcal{G}_s^s$ and cubic graphs in $\mathcal{G}_s^w$.

\begin{definition}For all $k \geq 2$, we denote by $H_k$ the graph consisting of a $2k$-gon with vertices
$v_1,\ldots,v_{2k}$ and a singly subdivided edge from $v_i$ to
$v_{i+3 \bmod 2k}$ for all $i$ even, i.e., adjacent to the $v_i$ there are $k$ additional degree $2$
vertices $u_1,\ldots,u_k$ and edges $u_iv_{2i}$ for all~$i \in
\{1,\ldots,k\}$, $u_iv_{2i+3}$ for all $i \in \{1,\ldots,k-2\}$,
$u_{k-1}v_1$ and $u_kv_3$. We observe that $H_k$ is planar if and
only if $k$ is even, see Figure~\ref{fig:K33} for a drawing of $H_3$.
\end{definition}

 \begin{figure}[htb]
  \centering
  \includegraphics{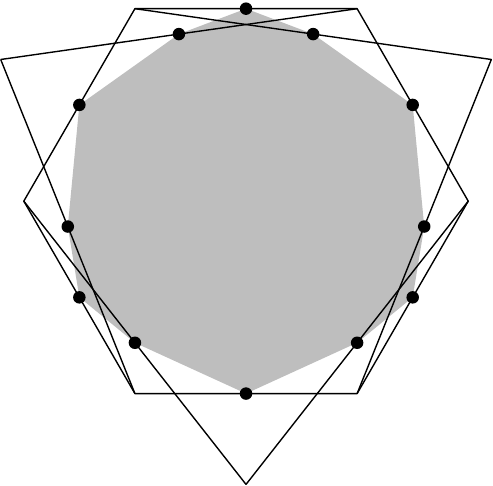}
  \caption{The graph $H_3$ drawn as in Theorem~\ref{thm:3over4}.}
  \label{fig:K33}
 \end{figure}

\begin{theorem}\label{thm:3over4}
 For every $k \geq 2$, $H_k \in \mathcal{G}_s^s$. In particular, for every $n \geq 9$
there is a non-planar $n$-vertex graph in $\mathcal{G}_s^s$.
\end{theorem}
\begin{proof}
We start by drawing $C_{2k}$ as a regular $2k$-gon. Take an edge $e
= xy$ and denote by $x',y'$ the neighbors of $x$ and $y$,
respectively. For convenience consider $e$ to be of horizontal slope
with the $2k$-gon below it. Our goal is to place $v_e$ a new vertex
and edges $v_e x'$, $v_e y'$ preserving the convexity of vertices
and midpoints of edges. We consider the upward ray $r$ based at the
midpoint $m_e$ of $e$ and the upward ray $s$ of points whose
$x$-coordinate is the average between the $x$-coordinates of $m_e$
and $x'$. We denote by $\Delta$ the triangle with vertices the
midpoint $m_{x'x}$ of the edge $x'x$, the point $x$ and $m_e$.  Since
$s \cap \Delta$ is nonempty, we place $v_e$ such that the midpoint
of $v_ex'$ is in $s \cap \Delta$.  Clearly $v_e$ is on $r$ and lies in the triangle defined by $xy$ and the lines supporting edges $x'x$ and $y'y$. Hence,
the middle point of $v_ey'$ is in the corresponding triangle
$\Delta'$ and the convexity of vertices and midpoints of edges is
preserved. See Figure~\ref{fig:4over3} for an illustration. Since we
only have to add a vertex on every other edge of $C_{2k}$, these
choices are independent of each other. It is easy to verify that the
constructed graph is $H_k$.

 \begin{figure}[htb]
  \centering
  \includegraphics{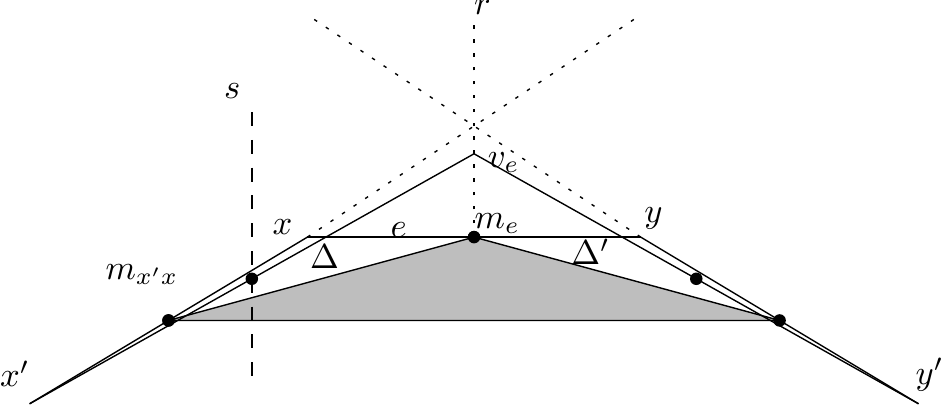}
  \caption{The construction in Theorem~\ref{thm:3over4}}
  \label{fig:4over3}
 \end{figure}

%We start by drawing $C_{2k}$ as a $2k$-gon with all interior angles
%equal and alternating long and short edges of lengths $s$ and
%$\ell$, respectively. For every long edge $e$ add a vertex $v_e$
%connected to the neighbors $x',y'$ of the endpoints $x,y$ of $e$.
%For convenience consider $e$ to be of horizontal slope with the
%$2k$-gon below it. We place $v_e$ on an upward ray $r$ based at the
%midpoint $m_e$ of $e$. We want that the midpoint $m$ of the edge
%from $v_e$ to $x'$ lies inside the triangle $\Delta$ defined by the
%midpoint $m_{x'x}$ of the edge $x'x$, the point $x$ and $m_e$. If
%$v_e$ is chosen very close to $m_e$, then $m$ lies below $\Delta$.
%In order to preserve strict convexity of vertices, the highest point
%$v_e$ can go on $r$ is just below the highest point $z$ of the
%isosceles triangle with basis $e$ and angles
%$\pi-\frac{(2k-2)\pi}{2k}=\frac{\pi}{k}$. In this case $m$ lies
%slightly below the midpoint $m_{x'z}$ of the line $x'z$. Since
%$\ell$ is chosen substantially larger than $s$, $m$ is above
%$\Delta$ in this case. Therefore there is a choice of $v_e$, such
%that $m$ is inside $\Delta$. See Figure~\ref{fig:4over3} for an
%illustration. Since we add a vertex $v_e$ only on every other edge
%of $C_{2k}$, these choices are independent of each other.

% \begin{figure}[htb]
%  \centering
%  \includegraphics{4over3.pdf}
%  \caption{The construction in Theorem~\ref{thm:3over4}}
%  \label{fig:4over3}
% \end{figure}

%It is easy to verify that the constructed graph is $H_k$.
\end{proof}

\begin{definition}For all $k \geq 3$, we denote by $P_k$ the graph consisting of a prism over a $k$-cycle. We observe that $P_k$ is
a $3$-regular graph.
\end{definition}

\begin{theorem}\label{thm:3over2}
For every $k \geq 3$, $P_k \in \mathcal{G}_s^w$. In particular, for
every even $n\geq 6$ there is a $3$-regular $n$-vertex graph in
$\mathcal{G}_s^w$.
\end{theorem}
\begin{proof}
Let $k\geq 3$. In order to draw $P_k$, place $2k$ vertices $v_0,
\ldots, v_{2k-1}$ as the vertices of a regular $2k$-gon in the plane. Add all
\emph{inner} edges of the form $v_iv_{i+ 2\bmod 2k}$ for all $i$ and
\emph{outer} edges $v_iv_{i+ 1\bmod 2k}$ for $i$ even. Clearly, the
midpoints of outer edges form a strictly convex set and their
convex hull is a regular $k$-gon. Now, consider four consecutive vertices 
in the boundary of the $2k$ gon, say
$v_0, \ldots, v_3$. They induce two outer edges, $v_0v_1$ and
$v_2v_3$ and two inner edges $v_0v_2$ and $v_1v_3$. Now, the
triangles~$v_0v_1v_2$ and $v_1v_2v_3$ share the base segment
$v_1v_2$. Hence, the segments~$m_{v_2v_3}m_{v_1v_3}$
and~$m_{v_2v_0}m_{v_1v_0}$ share the slope of $v_1v_2$. Now, since
the angle between $v_1v_2$ and $v_2v_3$ equals the angle between
$v_1v_2$ and $v_0v_1$ and $v_0v_1$ and $v_2v_3$ are of equal length,
the segment $m_{v_2v_3}m_{v_1v_0}$ also has the same slope. Thus,
all the midpoint lie on a line and all midpoints lie on the boundary
of the midpoints of outer edges. See Figure~\ref{fig:3over2} for an
illustration.
  \begin{figure}[htb]
  \centering
  \includegraphics{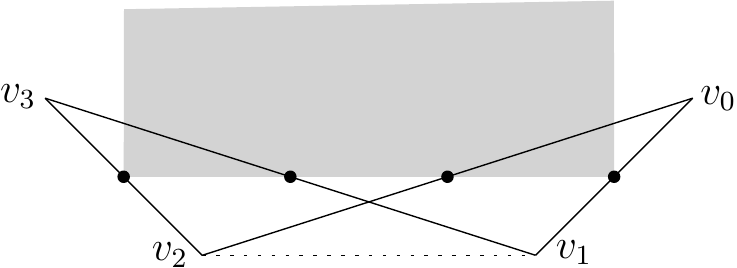}
  \caption{The construction in Theorem~\ref{thm:3over2}}
  \label{fig:3over2}
 \end{figure}

\end{proof}

We do not know of any 3-regular graphs in $\mathcal{G}_s^s$. More generally we believe that:

\begin{conjecture}
 If $G\in\mathcal{G}_s^s$ then $G$ is $2$-degenerate, i.e., every non-empty induced subgraph has a vertex of degree at most $2$.
\end{conjecture}

\subsection{Structural questions}

One can see, although it is tedious, that adding a leaf at the vertex $r_1$ of $L_8$ (see
Definition~\ref{def:Ln}) produces a graph not in $\mathcal{G}_s^w$.
Under some conditions it is possible to add leafs to graphs in
$\mathcal{G}_s^s$. We say that an edge is \emph{$V$-crossing} if it
intersects the interior of~$P_V$.

\begin{proposition}\label{lem:addleaf}
 Let $G\in\mathcal{G}_s^s$ be drawn in the required way. If $uv$ is not $V$-crossing, then attaching a new vertex $w$ to $v$ yields a graph in $\mathcal{G}_s^s$.
\end{proposition}
\begin{proof}
Let $G\in\mathcal{G}_s^s$ with at least $3$ vertices and let $e=uv$
be the edge of $G$ from the statement. For convenience consider that
$uv$ come in clockwise order on the boundary of $P_V$. Consider the
supporting line $H$ of $P_E$ through the midpoint $m_e$ of $e$,
whose side containing $P_E$ contains $v$. A new midpoint can go
inside the triangle $\Delta$ defined by $H$, the two clockwisely consecutive
supporting lines of $P_E$, both intersecting in a midpoint $m'$.
Since $P_E$ is contained in $P_V$ a part of $\Delta$ lies outside
$P_V$. Choosing the midpoint of a new edge attached to $v$ inside
this region very close to $e$ preserves strict convexity of vertices
and midpoints. See Figure~\ref{fig:addleaf} for an illustration.

 \begin{figure}[htb]
  \centering
  \includegraphics{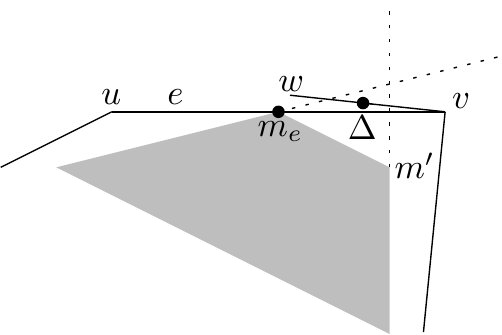}
  \caption{The construction in Proposition~\ref{lem:addleaf}}
  \label{fig:addleaf}
 \end{figure}
\end{proof}

We pose the following 
\begin{question}
Is the class $\mathcal{G}_s^s$ is closed under adding leafs?
\end{question}

Despite the fact that $K_{2,n} \notin \mathcal G_{s}^s$ for all $n \geq 3$, we have
found in Theorem~\ref{thm:3/2n-2lb} a subdivision of $K_{2,n}$ which
belongs to $\mathcal G_s^s$. Similarly, Theorem~\ref{thm:3over4} gives that a
subdivision of~$K_{3,3}$ is in $\mathcal G_s^s$, while $K_{3,3}$ is not. We
have the impression that subdividing edges facilitates drawings in
$\mathcal{G}_s^s$. Even more, we believe that:

\begin{conjecture}
 The edges of every graph can be (multiply) subdivided such that the resulting graph is in $\mathcal{G}_s^s$.
\end{conjecture}

\section{Minkowski sums}
We show that the largest cardinality of a weakly
convex set $X$, which is a subset of the Minkowski sum
of a convex planar $n$-point set $A$ with itself is $2n$. If $X$ is required to be strictly convex, then the largest 
size of such a set lies between $\frac{3}{2}n$ and $2n-2$.

As mentioned in the introduction there is a slight trade-off when translating the graph drawing problem to the Minkowski sum problem. Since earlier works have been considering only asymptotic bounds this was neglected. Here we are fighting for constants, so we deal with it. Recall that a point $x\in X\subseteq A+A$ is not captured by the graph model if $x=a+a$ for some $a\in A$. Indeed, the point $x$ corresponds to a vertex in the drawing of the graph. In order to capture the trade-off, for every $i,j \in \{s,w,a\}$, we
define $\widetilde{g}_i^j(n)$ as the maximum value of $n'+m$, where $m$ is the number of edges of an $n$-vertex graph in $\mathcal{G}_i^j$ and $n'$ of its vertices can be added to the set of midpoints in such a way that the resulting set is $\begin{cases} \mbox{strictly convex} &\mbox{if } j=s \\
\mbox{weakly convex} & \mbox{if } j=w \\
\mbox{arbitrary} & \mbox{if } j=a \end{cases}$.

\medskip

 We recall that a \emph{vertex~$v$ sees an edge $e$} if the
straight-line segment connecting $v$ and the midpoint $m_e$ of~$e$
does not intersect the interior of~$P_E$.

\begin{lemma}\label{lem:addvertex}
 Let $G\in \mathcal{G}_s^w$ be drawn in the required way and $v\in G$. If $v$ can be added to the drawing of $G$ such that $v$ together with the midpoints of $G$ is weakly convex, then every edge $vw\in G$ is seen by $w$.
\end{lemma}
\begin{proof}
 Otherwise the midpoint of $vw$ will be in the convex hull of $v$ together with parts of $P_E$ to the left and to the right of $vw$, see Figure~\ref{fig:addvertex}.
 \begin{figure}[htb]
  \centering
  \includegraphics{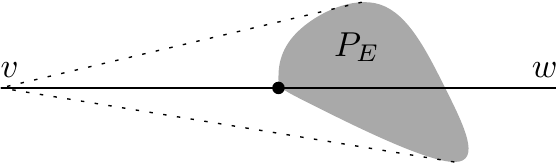}
  \caption{The contradiction in Lemma~\ref{lem:addvertex}}
  \label{fig:addvertex}
 \end{figure}
\end{proof}

We say that an edge is {\it good} if it
can be seen by both of its endpoints.

\begin{theorem}\label{thm:minkweak}
 For every $n\geq 3$ we have $\widetilde{g}_s^w(n)=2n$. That is, the largest
cardinality of a weakly convex set $X \subseteq A +
A$, for $A$ a convex set of $n$ points in the plane, is $2n$.
\end{theorem}
\begin{proof}
The lower bound comes from drawing $C_n$ as the vertices and edges
of a convex polygon. The set of vertices and midpoints is 
weakly convex.

For the upper bound let $G\in \mathcal{G}_s^w$ with $n$ vertices and
$m$ edges, we denote by~$n_i$ the number of vertices of $G$ that see
$i$ of its incident edges for $i\in\{0,1,2\}$. Since every edge is
seen by at least one of its endpoints and every vertex sees at most
$2$ of its incident edges (Lemma~\ref{sees2}), we know that $m = n_1
+ 2n_2 - m_g$, where $m_g$ is the number of good edges.

Let $n'$ be the number of vertices of $G$ that can be added to the
drawing such that together with the midpoints they are in weakly
convex position. Denote by~$n'_i$ the number of these vertices that
 see $i$ of its incident edges for $i\in\{0,1,2\}$. By Lemma~\ref{lem:addvertex} the edges seen by an
added vertex have to be good. Thus,
$m_g\geq\frac{1}{2}(n'_1+2n'_2)$. This yields
$$m+n'\leq n_1+2n_2-\frac{1}{2}(n'_1+2n'_2)+n'_0+n'_1+n'_2\leq
n_0+\frac{3}{2}n_1+2n_2\leq 2n.$$

\end{proof}

\begin{theorem}\label{thm:minkstrong}
 For every $n\geq 3$ we have $\lfloor\frac{3}{2}n\rfloor\leq \widetilde{g}_s^s(n)\leq 2n-2$. That is, the largest
cardinality of a convex set $X \subseteq A + A$, for
$A$ a convex set of $n$ points in the plane, lies within the above bounds.
\end{theorem}
\begin{proof}
 The lower bound comes from drawing $C_n$ as the vertices and edges of a convex polygon.
The set formed by an independent set of vertices and all
midpoints is in convex position.

Take $G\in \mathcal{G}_s^s$ with $n$ vertices and $m$ edges. The
upper bound is very similar to Theorem~\ref{thm:minkweak}. Indeed,
following the same notations we also get that $m = n_1 + 2n_2 -
m_g$. Again, the edges seen by an added vertex have to be good.
Since now moreover the set of addable vertices has to be
independent, we have $m_g\geq n'_1+2n'_2$. This yields $$m+n'\leq
n_1+2n_2-n'_1-2n'_2+n'_0+n'_1+n'_2\leq n+n_2-n'_2.$$

If $n+n_2-n'_2>2n-2$ then either $n_2 = n$ and $n_2' < 2$, or $n_2 =
n-1$ and~$n_2' = 0$. In both cases we get that $n' \leq 1$. By
Theorem~\ref{2n-2} we have $m\leq 2n-3$, then it follows that
$m+n'\leq 2n-2$.

\end{proof}

\section{Conclusions}
We have improved the known bounds on $g_s^s(n)$, the number of edges
an $n$-vertex graph of strong convex dimension $2$ can have. Still
describing this function exactly is an open problem. We believe that graphs in $\mathcal G_s^s$ have 
degeneracy $2$. However, confirming our conjecture
would not improve our bounds. Similarly, the exact largest cardinality
$\widetilde{g}_s^s(n)$ of a convex set $X \subseteq A
+ A$ for $A$ a convex planar $n$-point set, remains to be
determined. Curiously, in both cases we have shown that the correct
answer lies between $\frac{3}{2}n$ and $2n$. The more general family
$\mathcal{G}_s^w$ seems to be easier to handle, in particular
we have provided the exact value for both $g_s^w$ and
$\widetilde{g}_s^w$.

From a more structural point of view we wonder what graph
theoretical measures can ensure that a graph belongs to $\mathcal{G}_s^s$
or $\mathcal{G}_s^w$. None of these classes is contained in the class of planar graphs.
The class $\mathcal{G}_s^w$ is not closed
under adding leafs. We do not know if the same holds for
$\mathcal{G}_s^s$. Finally, we believe that subdividing a graph often 
enough ensures that it can be drawn in $\mathcal{G}_s^s$.

\bibliography{refs}

\providecommand{\bysame}{\leavevmode\hbox to3em{\hrulefill}\thinspace}
\providecommand{\MR}{\relax\ifhmode\unskip\space\fi MR }
\providecommand{\MRhref}[2]{%
  \href{http://www.ams.org/mathscinet-getitem?mr=#1}{#2}
}
\providecommand{\href}[2]{#2}
\begin{thebibliography}{1}

\bibitem{LowBound2010}
Ond{\v{r}}ej B{\'{\i}}lka, Kevin Buchin, Radoslav Fulek, Masashi Kiyomi, Yoshio
  Okamoto, Shin-ichi Tanigawa, and Csaba~D. T{\'o}th, \emph{A tight lower bound
  for convexly independent subsets of the {M}inkowski sums of planar point
  sets}, Electron. J. Combin. \textbf{17} (2010), no.~1, Note 35, 4.

\bibitem{UpBound2008}
Friedrich Eisenbrand, J{\'a}nos Pach, Thomas Rothvo{\ss}, and Nir~B. Sopher,
  \emph{Convexly independent subsets of the {M}inkowski sum of planar point
  sets}, Electron. J. Combin. \textbf{15} (2008), no.~1, Note 8, 4.

\bibitem{Gar-15}
Ignacio Garc{\'{\i}}a{-}Marco and Kolja Knauer, \emph{Drawing graphs with
  vertices and edges in convex position}, Graph Drawing and Network
  Visualization - 23rd International Symposium, {GD} 2015, Los Angeles, CA,
  USA, September 24-26, 2015, Revised Selected Papers, 2015, pp.~348--359.

\bibitem{Hal-07}
Nir {Halman}, Shmuel {Onn}, and Uriel~G. {Rothblum}, \emph{{The convex
  dimension of a graph.}}, {Discrete Appl. Math.} \textbf{155} (2007), no.~11,
  1373--1383.

\bibitem{Onn-04}
Shmuel {Onn} and Uriel~G. {Rothblum}, \emph{{Convex combinatorial
  optimization.}}, {Discrete Comput. Geom.} \textbf{32} (2004), no.~4,
  549--566.

\bibitem{SV2010}
Konrad~J. Swanepoel and Pavel Valtr, \emph{Large convexly independent subsets
  of {M}inkowski sums}, Electron. J. Combin. \textbf{17} (2010), no.~1,
  Research Paper 146, 7.

\bibitem{Tiw-14}
Hans~Raj {Tiwary}, \emph{{On the largest convex subsets in Minkowski sums.}},
  {Inf. Process. Lett.} \textbf{114} (2014), no.~8, 405--407.

\end{thebibliography}
\bibliographystyle{amsplain}

\end{document}